\documentclass [10.5pt]{amsart}
\usepackage{amsthm, amsmath, amssymb,latexsym, mathrsfs}
\usepackage[usenames,dvipsnames]{color}

\newtheorem{theorem}{Theorem}[section]
\newtheorem{lemma}[theorem]{Lemma}
\newtheorem{corollary}[theorem]{Corollary}
\newtheorem{proposition}[theorem]{Proposition}
\newtheorem{remark}[theorem]{Remark}
\newtheorem{definition}[theorem]{Definition}

\newcommand{\nc}{\newcommand} 
\nc{\cH}{{\mathcal H}}
\nc{\cA}{{\mathcal A}}
\nc{\cG}{{\mathcal G}}
\nc{\cC}{{\mathcal C}}
\nc{\cO}{{\mathcal O}}
\nc{\cI}{{\mathcal I}}
\nc{\cB}{{\mathcal B}}
\nc{\cY}{{\mathcal Y}}
\nc{\cK}{{\mathcal K}} 
\nc{\cX}{{\mathcal X}}
\nc{\cS}{{\mathcal S}}
\nc{\cE}{{\mathcal E}}
\nc{\cF}{{\mathcal F}}
\nc{\cZ}{{\mathcal Z}}
\nc{\cQ}{{\mathcal Q}}
\nc{\cN}{{\mathcal N}}
\nc{\cP}{{\mathcal P}}
\nc{\cL}{{\mathcal L}}
\nc{\cM}{{\mathcal M}}
\nc{\cT}{{\mathcal T}}
\nc{\cW}{{\mathcal W}}
\nc{\cU}{{\mathcal U}}
\nc{\cJ}{{\mathcal J}}
\nc{\cV}{{\mathcal V}}
\nc{\bH}{{\mathbb H}}
\nc{\bA}{{\mathbb A}}
\nc{\bG}{{\mathbb G}}
\nc{\bC}{{\mathbb C}}
\nc{\bO}{{\mathbb O}}
\nc{\bI}{{\mathbb I}}
\nc{\bB}{{\mathbb B}}
\nc{\bY}{{\mathbb Y}}
\nc{\bK}{{\mathbb K}} 
\nc{\bX}{{\mathbb X}}
\nc{\bS}{{\mathbb S}}
\nc{\bE}{{\mathbb E}}
\nc{\bF}{{\mathbb F}}
\nc{\bZ}{{\mathbb Z}}
\nc{\bQ}{{\mathbb Q}}
\nc{\bN}{{\mathbb N}}
\nc{\bP}{{\mathbb P}}
\nc{\bL}{{\mathbb L}}
\nc{\bM}{{\mathbb M}}
\nc{\bT}{{\mathbb T}}
\nc{\bW}{{\mathbb W}}
\nc{\bU}{{\mathbb U}}
\nc{\bD}{{\mathbb D}}
\nc{\bJ}{{\mathbb J}}
\nc{\bV}{{\mathbb V}}
\nc{\bbZ}{{\mathbb Z}}
\nc{\bR}{{\mathbb R}}
\nc{\fr}{{\rightarrow}}
\nc{\co}{{\nabla}}

\nc{\cu}{{\barline{\nabla}}}

\usepackage{tikz-cd}

\usepackage{bm}

\begin{document}
\title {Trigonal deformations of rank one
and Jacobians}

\author
{Valentina Beorchia, Gian Pietro Pirola and Francesco Zucconi}

\address
{Dipartimento di Matematica e Geoscienze,
Dipartimento di Eccellenza 2018-2022,
  Universit\`a di Trieste,
via Valerio 12/b, 34127 Trieste, Italy,
ORCID ID 0000-0003-3681-9045,
\it{beorchia@units.it}}

\address
{Dipartimento di Matematica, 
Dipartimento di Eccellenza 2018-2022,
 Universit\`a di Pavia,
Via Ferrata 1,  
27100 Pavia, Italy,
\it{gianpietro.pirola@unipv.it}}

\address
{Dipartimento di Scienze Matematiche, Informatiche e Fisiche, 
Universit\`a degli studi di Udine,
33100 Udine, Italy,
\it{francesco.zucconi@uniud.it}}

\maketitle

\begin{abstract} 
In this paper we study the infinitesimal deformations of a trigonal curve that preserve the trigonal series and such that the associate infinitesimal variation of  Hodge structure (IVHS) is of rank $1.$  We show that if $g\geq 8$ or $g=6,7$ and the curve is Maroni general, this locus is zero dimensional. Moreover, we complete the result \cite[Theorem 1.6]{NP}. We show in fact that if $g\geq 6$, the hyperelliptic locus $\cH_g$  is the only $2g-1$-dimensional sub-locus $\cY$
of the moduli space $\cM_g$ of curves of genus $g$, such that for the general element $[C]\in \cY$, its Jacobian $J(C)$ is dominated by a hyperelliptic Jacobian of genus $g'\geq g$.
\end{abstract}

\section*{Introduction}

Let $\cT_g$ be the locus in the moduli space $\cM_g$ of smooth curves of genus $g$ given by points $[C]\in\cM_g$ such that $C$ is a trigonal curve, that is $C$ admits a triple covering of $\mathbb P^1$ and $C$ is not hyperelliptic.

In this paper we study infinitesimal deformations which come from families of trigonal curves and in particular
the ones having the associated infinitesimal variation of Hodge structure (IVHS) of rank $1.$ 

Our main motivation was to complete the characterization of $2g-1$-dimensional families $\mathcal Y$ of curves with Jacobian family dominated by a hyperelliptic Jacobian family. 
Indeed in \cite[Theorem 1.6]{NP} it is shown that if $\mathcal Y$ is a $2g-1$-dimensional closed irreducible subvariety of $\cM_g$ with $g\geq 5$, such that the Jacobian of its generic element is dominated by a hyperelliptic Jacobian, 
then $\mathcal Y\subset \cH_g$, where $\cH_g$ denotes the hyperelliptic locus, or $\mathcal Y\subset \cT_g$. In this paper we rule out the trigonal case. Our main result is the following:
\medskip

\noindent
\begin{theorem} \label{main}
If $g\geq 6$ then $\cH_g$ is the unique closed irreducible subvariety ${\mathcal Y}\subset\cM_g$ of dimension $2g-1$ such that for its generic element $[C]\in {\mathcal Y}$ there exists $[D]\in\cH_{g'}$ such that $J(D)\twoheadrightarrow J(C)$.
\end{theorem}
\medskip

We point out that, in principle, we have to consider all the codimension-$2$ subvarieties of $\cT_g.$
The argument goes as follows. If $C$ is smooth and trigonal the Babbage-Enriques-Petri theorem gives that the intersection of all the quadric that contain the canonical image of $C$ is a ruled surface $S\subset\bP^{g-1}.$ This surface $S$ can be also embedded, by an extension $\kappa_2: S\to \bP H^1(T_C)$ of the bicanonical morphism of $C$,  in the projective space $\bP H^1(T_C)$, where $T_C$ is the tangent sheaf of $C$. We identify $H^1(T_C)$ with the tangent space of $\cM_g$ at $[C]$, by possibly adding a level structure or by considering the moduli stack
$\mathscr{M}_g$, if $[C]$ is a singular point of the moduli space. By the work of Griffiths, the points of the surface $\kappa_2(S)$ correspond to the locus of infinitesimal deformations with IVHS of rank one (see \cite {Gri1}) and Lemma \ref{Ruledsurfacerank1}).

 Now, in \cite[Proof of Theorem 1.6, p.13 ]{NP}, it is shown that, under the conditions of theorem (\ref{main}), if ${\mathcal Y}\subset\cT_g$, then for the generic $[C]\in \mathcal Y$ there must exist a non-degenerate rational curve $Z\subset S\subset \bP^{g-1}$ which corresponds to a $1-$dimensional component of the intersection $\Gamma=\kappa_2(S)\cap\mathbb P(T_{\cT_{g},[C] })\subset \bP H^1(T_C)$, where $T_{\cT_{g},[C] }$ is the tangent space to the trigonal locus at $[C]$. We show that $\Gamma$ cannot contain such a $Z$, by analysing the geometry on the Hirzebruch surface $\bF_n$ isomorphic to $S$. We recall that the integer $n\ge 0$ is classically called the Maroni invariant of $C$; we shall perform our study in terms of the Maroni degree $k={(g-2-n) \over 2}$ for trigonal curves, 
 that is in terms of the degree of a non positive section in the extended canonical embedding in $\bP^{g-1}$, see subsection \ref{permaroni}. We show:
 \begin{theorem} \label{k2}
 If $g\geq 8$, or $6\leq g\leq 7$ and Maroni degree $k=2$, or $g=6$ and $k\neq 1$, then $\Gamma$ has dimension $0$.
 \end{theorem}
The proof will follow from Proposition \ref{casosette} and Proposition \ref{casosei}. 

In the special cases $g=7,6$ and $k=1$ we need to deepen our analysis. We consider the locus $\cT^{1}_{g}\subset \cT_{g}$ of trigonal curves with Maroni degree $k=1$ and we define $\Gamma^1\subset \Gamma$ the intersection $\Gamma\cap \bP(T_{\cT^{1}_{g},[C]})$. We show:

\begin{theorem} \label{k1}
If $g=6$ or $g=7$ and Maroni degree $k=1$ and $C$ is generic, then $\Gamma$ is irreducible and $\Gamma^{1}$ has dimension $0$. 

Moreover, if $\Gamma$ is reducible, then $\Gamma^{1}$ does not contain any rational curve $Z_1$, such that $\kappa_2 ^{-1} (Z_1)\subset \bP^{g-1}$ is nondegenerate.
\end{theorem}

The result will be a consequence of Proposition \ref{ilcasettosette} for $g=7$
and Proposition \ref{ilcasettosei} if $g=6.$ 

From the two theorems we deduce that there are no non-degenerate rational curve $Z\subset \Gamma$, which
implies our main result.
\medskip

Our point of view is also related to problems concerning the relative irregularity of families of curves, see
for instance \cite{BGN},\cite {FNP},\cite{GST}, and \cite{Gon}, since our result gives an infinitesimal version of a result of Xiao \cite[Corollary 4, page 462]{X}. To shorten the paper we do not include this here as well as some partial computation on genus $5$ case, but we only finally remark that hopefully, this can shed some new light on the slope problem as treated in \cite{BZ1} and \cite{BZ2}.

\medskip
\noindent {\bf Acknowledgments:}
 The authors are grateful to Juan Carlos Naranjo for some useful conversations.
 
\medskip

The authors are members of GNSAGA of INdAM.

The first and third authors are supported by national MIUR funds, 
PRIN project {\it Geometria delle variet\`a algebriche} (2015). 

The first author is also supported by national MIUR funds
{\it Finanziamento Annuale Individuale delle Attivit\`a Base di Ricerca} - 2018.

The second author is partially supported by PRIN  project {\it Moduli spaces and Lie theory} (2015) and by MIUR: Dipartimenti di
Eccellenza Program (2018-2022) - Dept. of Math. Univ. of Pavia.

 The third author is supported by Universit\`a degli Studi di Udine - DIMA project 
 {\it Geometry PRID ZUCC 2017}.

  \medskip
  
 \noindent {\bf 2010 Mathematics Subject Classification: 14J10, 14D07.}


\section{A Gaussian lemma}
Let $C$ be a smooth curve and let $L$ be line bundle of degree $d$ on $C$. Consider a base point free linear system subspace $V\subset H^0(C,L)$. We can associate two objects with $V$, the Gaussian section and the Euler class, which we recall.

Let $f\colon C\to { \bP (V^{\vee})} \equiv \bP^r$ be the projective morphism given by $V$ and
\begin{equation}\label{diff}
 0\to T_C\stackrel{df}\rightarrow f^{\star}(T_{\bP^r})\to N_f\to 0
 \end{equation}
the exact sequence associated with its differential. We tensor it by the canonical sheaf $\omega_C$: 
 \begin{equation}\label{gaussian section}
 0\to \cO_C\stackrel{}\rightarrow \omega_C\otimes  f^{\star}(T_{\bP^r})\to \omega_C\otimes N_f\to 0
 \end{equation}
 Then we can consider the section 
 $$
 \Omega_V\in H^{0}(C, \omega_C\otimes  f^{\star}(T_{\bP^r}))
 $$ 
 obtained as the image of $1\in H^0(C,\cO_C)$, and represents the differential $d\ f$.
 \begin{definition} We call $\Omega_V$ the Gaussian section of the morphism $f\colon C\to\bP^r$.\end{definition}
\noindent We want to relate $\Omega_V$ to the Euler class whose definition we briefly recall.

It is well known that $H^1(\bP^r,\Omega^1_{\bP^r})=\mathbb C$ and that the Euler sequence
 $$
 0\to\cO_{\bP^r}\to (\cO_{\bP^r}(1))^{\oplus r+1}\to T_{\bP^r}\to 0
 $$
 is given by a non trivial class $\eta_{{\rm{eul}}}\in H^1(\bP^r,\Omega^1_{\bP^r})$. The pull-back to $C$ of the Euler sequence gives
 \begin{equation}\label{eulerosuC}
 0\to \cO_C\to V\otimes L \to f^{\star}(T_{\bP^r})\to 0.
 \end{equation}

 The extension class of the sequence (\ref{eulerosuC}) determines an element $\eta_{V}\in H^1(C,f^{\star}(\Omega^1_{\bP^r}))$: we call it the {\it{Euler class of $f\colon C\to |V|^\vee$.}}

Following \cite[Page 804]{ACG}, we recall that to a line bundle $L$ on a smooth curve $C$ we can associate a sheaf
$\Sigma_L$, determined by the Chern class $c_1(L)\in H^1(C,\omega_C)\cong Ext^1(T_C,\cO_C)$. Let
\begin{equation}\label{sheavesofdiff}
0\to \cO_C\to \Sigma_L\stackrel{\tau}{\rightarrow} T_C\to 0
\end{equation} 
be the extension determined by $c_1(L)$.

By taking into account the sequence (\ref{diff}), we get a commutative diagram

\begin{equation} \label{diagram}
\begin{array}{ccccccc}
&   &&       0&& 0&\\
 &   &&       \downarrow&& \downarrow&\\
0\to &\cO_C     &\to& \Sigma_L        &\stackrel{\tau}{\rightarrow}& T_C &\to 0\\
    &\vert\vert &   &\downarrow      &                     &\downarrow \ \Omega_V\\
 0\to  &\cO_C  &\to &V\otimes L & \to                  & f^{\star}(T_{\bP^r}) &\to 0\\
 &   &&       \downarrow&& \downarrow&\\
  &   &&       \cN_{f} &=& \cN_{f}&\\
  &   &&       \downarrow&& \downarrow&\\
  &   &&       0&& 0&\\
 \end{array}
 \end{equation}
and we observe that the differential $df$ determines a map of complexes, and the image of the 
bottom extension is the upper extension.

 \begin{lemma}\label{gaussianlemma} {\bf{ (Gaussian Lemma)}}  
 Consider the natural map given by the cup product and duality
$$
H^{0}(C, \omega_C\otimes  f^{\star}(T_{\bP^r})) \times H^1(C,f^{\star}(\Omega^1_{\bP^r})) \to H^1 (C, 
\omega_C).
$$
 Then we have
 $$
 \Omega_V \cdot \eta_V = c_1(L)\neq 0
 $$

As a consequence, the Gaussian section $\Omega_V$ does not belong to the image of the map
$$
H^0(C, V\otimes  \omega_C\otimes  L)\to H^0(C, f^{\star}(T_{\bP^r})\otimes \omega_C).
$$
 \end{lemma}

\begin{remark} We will use Lemma \ref{gaussianlemma} in the case where $L$ induces a $g^1_3$ on $C$. We stress that if $r=1$ then the ramification scheme $R_V$ of $f\colon C\to\bP^1$ is the zero locus of the Gaussian section, which is a section 
$$
\Omega_V \in H^0(C,\omega_C\otimes L^{\otimes 2}).
$$

\end{remark}

 \section{ Trigonal curves and Gaussian sections}
 \label{permaroni}
 Let $C\subset \bP^{g-1}$ be a canonical trigonal curve of genus $g\geq 5$ which has no $g^{2}_{5}$ and let $I_C$ be its graded ideal. 
By the Babbage-Enriques-Petri theorem the hyperquadrics containg $C$ intersect in
 a smooth rational ruled surface $S$ of minimal degree $g-2$ in $\bP^{g-1}$ and the trigonal series is cut on $C$ by the ruling of $S$. 
 
 Let us fix some notations and recall some known results concerning the Hirzebruch surfaces ${\mathbb F}_n=\bP(\cO_{\bP^1}\oplus\cO_{\bP^1}(n))$.
 The Picard group of ${\mathbb F}_n$ satisfies ${\rm{Pic}}({\mathbb F}_n)=[B]\mathbb Z\oplus [R]\mathbb Z$, where $B$ is a section of minimal self-intersection and $R$ a fiber in the ruling of the projection $\pi: {\mathbb F}_n\to\bP^1$.
 The basic intersection formulae are:
\begin{equation}\label{intersectionmatrix}
B^2=-n, \ BR=1, \ R^2=0.
\end{equation}

\begin{definition}\label{maronidegree}
Let $C$ be a trigonal curve of genus $g\geq 6$ and let $L$ be the line bundle of degree $3$ computing the unique trigonal series. We set $V:=H^0(C, L )$. The Maroni degree $k\in\mathbb N$ of $C$ can be characterized as the unique number such that
$$
h^0(C, L^{\otimes k+1})=k+2, \qquad h^0(C, \omega_C\otimes L^{\otimes k+2})>k+3.
$$
\end{definition}
The following bounds on $k$ have been established by Maroni \cite{M} and are well known
 \begin{equation}\label{disuguaglianze di maroni}
 \frac {g-4}{3}\leq k \leq \frac {g-2}{2}.
 \end{equation}
 It turns out that $S$ is an embedding of the Hirzebruch surface ${\mathbb F}_{g-2-2k}$ via the linear system 
 $|H|= |B+(g-2-k)R|$. 
 Moreover we have
 \begin{equation}\label{chiave}
 C\in |3B+(2g-2-3k)R|, \qquad K_{C}=H_{|C}.
 \end{equation}

We recall that $H^0({\mathbb F}_n,\cO_{{\mathbb F}_n}(R))\cong H^0(\bP^1,\pi_\star \cO_{{\mathbb F}_n}(R))\cong H^0(\bP^1,\cO_{\bP^1}(1))$.
We recall that a class $M=kB+sR$ is ample if and only if $MR=k>0$ and $MB=-nk+s>0$, that is $s>nk$,
and respectively nef if $s\geq nk\geq 0$. Finally, if $s\geq nk> 0$, then $M$ is big and nef.

The following results will be useful in the sequel:
\begin{lemma}\label{mu} 
If $m\geq 0$ and $s\geq n(m+1)+1$, then the multiplication map
\begin{equation} \label{molt}
\mu\colon H^0({\mathbb F}_n,\cO_{{\mathbb F}_n}(R))\otimes H^0({\mathbb F}_n, \cO_{{\mathbb F}_n}(mB+sR))\to H^0({\mathbb F}_n,\cO_{{\mathbb F}_n}(mB+(s+1)R))
\end{equation}
is surjective.
\end{lemma}
\begin{proof} Set $V:= H^0({\mathbb F}_n,\cO_{{\mathbb F}_n}(R))$. We tensor the sequence
\begin{equation}\label{solita}
 0\to \cO_{{\mathbb F}_n}(-R)\to V\otimes \cO_{{\mathbb F}_n}\to \cO_{{\mathbb F}_n}(R)\to0
 \end{equation}
 by $\cO_{{\mathbb F}_n}(mB+sR)$ where $m\ge 0$ and $s\geq n(m+1)+1$. By the free pencil trick the cokernel of $\mu$ injects
 into $H^1({\mathbb F}_n, \cO_{{\mathbb F}_n}(mB+(s-1)R))$ which is the dual of 
 $ H^1({\mathbb F}_n, \cO_{{\mathbb F}_n}(-(m+2)B-(s+n-1)R))$ by Serre duality and the fact that the canonical divisor is $K_{{\mathbb F}_n}\sim -2B-(2+n)R$. By the hypotheses
 on $m$ and $s$, the divisor $M:=(m+2)B+(s+n-1)R$ is big and nef, therefore $H^1({\mathbb F}_n, \cO_{{\mathbb F}_n}(-M))=0$.
 \end{proof}

\begin{proposition} \label{mu1}
The map $\mu: H^0(S,H+R)\otimes H^0(S,R)\to H^0(S,H+2R)$ is surjective.
\end{proposition}
\begin{proof} As in the proof of Lemma (\ref{mu}) we only need to show that $H^1(S,\cO_S(H ))=0$. By the
remarks above we have $H\sim K_S+C$. It is easy to show that $C$ is $1$-connected, $|nC|$ gives a morphism to a surface. We conclude by applying Ramanujam's vanishing theorem\cite{R}.
\end{proof}

Now consider the exact sequence:
$$
0\to\cO_S(H+2R-C)\to\cO_S(H+2R)\to \omega_C(2L)\to 0
$$
and the restriction map 
$$
r\colon H^0(S,\cO_S(H+2R))\to H^0(C,\omega_C(2L)).
$$
Now let us observe that in the trigonal case, the Gaussian section 
$\Omega_V \in H^0(C,\omega_C\otimes L^{\otimes 2})$.

\begin{corollary}\label{grandeGianPietro}
The Gaussian section does not belong to the image of $r$.
\end{corollary}
\begin{proof}
By contradiction assume that $\Omega_V=r(\alpha)$. By Proposition \ref{mu1} we have $\alpha=s_1\beta_1+s_2\beta_2$
which would imply 
that $\Omega_V$ belongs to the image of the multiplication map $\mu: H^0 (C,L)\otimes H^0(C,\omega_C\otimes L)\to H^0(C,\omega_C\otimes 2L)$. This is in contradiction with the Gaussian Lemma \ref{gaussianlemma}.
\end{proof}


\section{The locus of rank $1$ infinitesimal deformations }
Let $I_2:=I_{C}(2)$ be the degree two part of the homogeneous ideal of a trigonal canonical curve $C\subset
\mathbb P ^{g-1}$, and let $S$ be the ruled surface.
By a simple computation we see that 
\begin{equation}\label{isoquadr}
|2H|\cong |2H_{|C}|\cong |2K_C|.
\end{equation}
If we consider the bicanonical map 
$$
 C\to |2K_C|^\vee \cong\bP (H^1(C,T_C))=\bP^{3g-4},
$$ 
we observe that it extends to an embedding $\kappa_2: S\to \bP^{3g-4}$. 
Since $C$ is not hyperelliptic, the multiplication map 
$$
\mu\colon {\rm{Sym}}^{2}H^0(C,\omega_C)\to H^0(C,\omega_C^{\otimes 2})
$$ 
is surjective. By definition the kernel of $\mu$ is $I_2$. By Serre duality we obtain:
\begin{equation}\label{la2conormalita}
 0\fr H^1(C, T_C)\stackrel{\iota}{\to} {\rm{Sym}}^2(H^1(C,\cO_C))\to I_2^{\vee}\to 0.
 \end{equation}
The inclusion 
 $\iota\colon H^1(C, T_C)\to {\rm{Sym}}^2(H^1(C,\cO_C))$ is given by $\xi\stackrel{\iota}{\mapsto}q_{\xi}$ where $q_{\xi}$ is the quadric associated with the co-boundary homomorphism 
 $$
 \partial_{\xi}\colon H^0(C,\omega_C)\to H^0(C,\omega_C)^{\vee}=H^1(C,\cO_C)
 $$ 
 of the extension class $\xi\in H^1(C, T_C)={\rm{Ext}}^{1}_{\cO_C}(\omega_C,\cO_C)$; see \cite{Gri1}. 
 \begin{definition}\label{ilrango}
 We define {\it{the rank of $\xi$}} as the rank of its associated quadric $q_{\xi}$.
 \end{definition}

By the standard properties of the Veronese embedding 
$$
\nu_2\colon \bP(H^1(C,\cO_C))\to \bP({\rm{Sym}}^2(H^1(C,\cO_C)))
$$ 
it follows that 
$$
\kappa_2 (S)=\nu_2(\bP(H^1(C,\cO_C))\cap \iota(\bP(H^1(C,T_C))\subset \bP({\rm{Sym}}^2(H^1(C,\cO_C))).
$$
We have:
\begin{lemma}\label{Ruledsurfacerank1} The image of the embedding of $\kappa_2: S\to \bP^{3g-4}$ satisfies:
$$
\kappa_2 (S)=\{[\xi]\in \bP H^1(C,T_C)\mid \partial_{\xi}\colon H^0(C,\omega_C)\to H^1(C,\cO_C)\ \text{has rank 1}\};
$$
\end{lemma}
\begin{proof} It follows from the sequence (\ref{la2conormalita}) and its dual. See also: \cite[p. 271]{Gri1}. 
\end{proof}

\subsection{Trigonal deformations of rank $1$}
In the present section we shall study the locus of $S$ corresponding to deformations which preserve the 
property of having a trigonal series.

 As in the Introduction we denote by $\cT_g$ the trigonal locus and let $\cT_{g}^{k}\subset \cT_g$ be the locus of trigonal curves with Maroni degree $k$. 
 We define
 $$
 T^{k}_{C}:= T_{\cT_g^k,[C]}\subseteq H^{1}(C,T_C), \qquad T:=T_{\cT_g,[C]}\subseteq H^{1}(C,T_C).
 $$ the tangent spaces to respectively $\cT_{g}^{k}$ and $\cT_{g}$ at $[C]$.

Consider the natural homomorphism 
 $$
 H^1(C,T_C)\otimes H^0(C,\omega_C\otimes L^{\otimes 2})\to H^1(C, L^{\otimes 2})
 $$
 given by the cup-product $\xi\otimes \sigma\mapsto \xi\cdot\sigma$. It holds:
 
 \begin{lemma}\label{lospaziotangentetrigonale} 
 If $[C]\in \cT_g$ then 
 \begin{equation}\label{il luogo tangentetrigonale}
T=\{\zeta\in H^1(C,T_C): \zeta\cdot \Omega_V=0\in H^1(C,2L)\}.
\end{equation} 
 \end{lemma}
 
 \begin{proof} Let $\cC_{\xi}\to{\rm{Spec}}(\mathbb C[\epsilon]/(\epsilon^2))$ be the infinitesimal family associated with $\xi$. Since $\xi\in T$, the trigonal morphism $f\colon C\to\bP^1$ lifts to $\cC_{\xi}$. By standard arguments it follows that the Gaussian section lifts to $\cC_{\xi}$. This implies that the cup product $\xi\cdot \Omega_V=0$. Since $h^1(C,L^{\otimes 2})=g-4$ and since the cup product $\cdot \Omega_V\colon H^1(C,T_C)\to H^1(C, L^{\otimes 2})$ is easily seen to be surjective by dualizing the map, we have that 
 $${\rm{dim}}_{\mathbb C}\{\zeta\in H^1(C,T_C): \zeta\cdot \Omega=0\in H^1(C,2L)\}= 2g+1.
 $$ 
 On the other hand ${\rm{dim}}_{\mathbb C}\cT_g=2g+1$. Then the claim follows.
\end{proof}

\begin{remark} \label{r}
Dually, the annihilator of $T$ in $H^1(C,T_C)^\vee \cong H^0(C, \omega_C^{\otimes 2})$ 
is $H^0( C, \omega_C^{\otimes 2}(-R_V) )$, where $R_V$ is the ramification divisor.

By taking into account the isomorphisms in
(\ref{isoquadr}), such a space corresponds to the subspace
$$
H^0(S, \cO_S (2H)\otimes\cI _{R_V,S}),
$$
where $\cI _{R_V,S}$ is the ideal sheaf of the subscheme $R_V \subset S$.

We shall denote by
$$
\Lambda := \bP (  H^0(S, \cO_S (2H)\otimes\cI _{R_V,S}) ).
$$
\end{remark}
Next we would like to determine the intersection of $\bP(T)$ with the surface $\kappa_2 (S)$.

\begin{definition}\label{luogo} 
The locus $\Gamma =\bP(T)\cap\kappa_2 (S)$ is called the locus of trigonal deformation of rank $1$.
\end{definition}

 \subsection{Proof of Theorem \ref{k2}}

In what follows with $k$ we shall always denote the Maroni degree.
For the reader's convenience we fix the relations and classes in ${\rm{Pic}}(S)={\rm{Pic}}(\mathbb F_{g-2-2k})$,
which we shall use:

 \begin{itemize}
 \item $R$ is the ruling inducing the trigonal linear system on $C$;
 \item $B$ is a section of minimal self-intersection of $\mathbb F_{g-2-2k}$, and $B^2=2k+2-g$;
 \item $A=B+(g-2-2k)R$ is the tautological divisor of $\mathbb F_{g-2-2k}$;
 \item $K_S=-2B-(g-2k)R$;
 \item $H=B+(g-2-k)R$ is the hyperplane divisor of the canonical embedding, $H_{|C}\sim K_C$, and $H^2=g-2$;
 \item $C\in |3B+(2g-2-3k)R|$.
 \end{itemize}

Recalling the bounds on $k$ given in (\ref{disuguaglianze di maroni}), we shall distinguish two cases: 
$g=2k+2$ or $2k+< g \le 3k+4$.

\subsection{The case $g=2k+2$ and $k \ge 2$: curves of even genus with Maroni invariant zero}

We shall show that in this case the subscheme $R_V\subset S$ is a complete intersection of two divisors
$G_1, G_2 \sim 2B+(k+2)R$, and since $h^0(S,\cO_S(2H-G_i))\neq 0$, they determine a pencil in $\Lambda$ with base locus exactly $R_V$; as a consequence we will obtain that the base locus ${\rm{Bs}}(\Lambda)$ of
$\Lambda$ satisfies
${\rm{Bs}}(\Lambda)=R_V$.

Let $C$ be a trigonal curve of genus $g=2k+2$. In this case we have that 
$$
S\cong\mathbb P^1\times \mathbb P^1, \quad C\in |3B+(k+2)R|, \quad \Lambda\subseteq |2H|=|2B+2kR|.
$$

\begin{proposition}\label{casoparipari} 
If $g=2k+2$ with $k \ge 2$, then the base locus of $\Lambda$ satisfies $\dim {\rm Bs}(\Lambda)=0$ and ${\rm Bs}(\Lambda)=R_V$.
\end{proposition}

\begin{proof}
We observe that
$$
h^0(S, \cO_S(2B+(k+2)R) \otimes \cI_{R_V,S})\neq 0.
$$
Indeed, we have
$$
(2B+(k+2)R)_{|C} \sim H_{|C}+2R_{|C}+B_{|C} \sim R_V +B_{|C},
$$
and since $-C+2B+(k+2)R\sim - B$, we have $h^0(  S,\cO_S( -C+2B+(k+2)R ))=h^1(S,\cO_S( -C+2B+(k+2)R ))=0$, so there is an isomorphism
\begin{equation}\label{ppp}
H^0(\cO_{S}(2B+(k+2)R))\cong H^0( \cO_{C}(K_C+2L+B_{|C}) ).
\end{equation}
We consider in particular the pencil 
$R_V+|B_{|C} |$ in $|K_C+2L+B_{|C}$, and the corresponding pencil $|G|$ in
$|2B+(k+2)R|$ under the isomorphism (\ref{ppp}). By construction we have the base locus ${\rm Bs}|G|\supseteq R_V$. We claim that 
$$
{\rm Bs} |G| =R_V.
$$
We first show that ${\rm Bs} |G|$ contains no divisors. If by contradiction it contains a divisor $\Gamma$ in some linear system $|aB+bR|$ with 
\begin{equation}\label{ab}
0\le a \le 2, \qquad 0\le b \le k+2,
\end{equation}
then we must have
$$
\Gamma|{_C} < {\rm Bs} |K_C+2L+B_{|C} - R_V|=R_V,
$$
hence
\begin{equation}\label{first}
\Gamma \cdot C =3b+(k+2)a\le {\rm deg}\ R_V= 2g+4=4k+8.
\end{equation}
On the other hand, the base points of the moving part of the pencil must contain the residual part of $R_V$,
hence we must have
$$
\Gamma \cdot C + (G-\Gamma)^2 \ge {\rm deg}\ R_V=4k+8.
$$
This gives the inequality
\begin{equation}\label{diseq}
3b+(k+2)a+2(2-a)(k+2-b)\ge 4k+8.
\end{equation}
We observe that if $a=0$, then (\ref{diseq}) gives $b=0$.

If $a=1$, then by (\ref{diseq}) and (\ref{ab}) we get
$b=k+2$; in this case we would have $\Gamma \sim H+2R$ and $\Gamma|{_C}=R_V$, but this is not the case
since by Corollary \ref{grandeGianPietro} the Gaussian section 
$\Omega\not\in {\rm{Image}}(H^0(S,\cO_S(H+2R))\to H^0(C,\omega_C(2L))$. 

Finally, observe that the case $a=2$ can not occur, as by construction the moving part of the pencil
$R_V+|B_{|C} |$ is not cut out by fibers $R$ of the ruling.

Therefore $\dim {\rm Bs}(|G|)=0$. Since $G^2 =4k+8={\rm deg} \ R_V$, we get the statement of the Proposition.

\end{proof}

\subsection{The case $2k+2< g \le 3k+4$, $k\ge 1$ and $g\ge 6$} 

The argument is similar to the one of the previous section. We shall show that in this case the subscheme $R_V\subset S$ is a complete intersection of two divisors
$Q_V \sim 2B+(g-k)R$ and $Q_1\sim 2B+(2g-3k-2)R$, and since $H\sim B+(g-2-k)R$, the bounds (\ref{disuguaglianze di maroni}) imply $h^0(S,\cO_S(2H-Q_V))\neq 0$ and with the additional hypothesis $k\ge 2$ we have also 
\begin{equation}\label{eccezioni}
h^0(S,\cO_S(2H-Q_1))\neq 0.
\end{equation}
As a consequence we will obtain that the base locus of $\Lambda$ is
${\rm{Bs}}(\Lambda)=R_V$.

Finally, we shall prove that in the remaining cases
$g=6, \ g=7$ and  $k=1$,
the linear system $\Lambda$ has a one dimensional fixed component.

Since we are assuming $g>2k+2$, the ruled surface $S\cong\mathbb F_{g-2k-2}$ admits a negative section $B^2<0$. 

\begin{lemma}\label{primaquadrica} 
There exists a unique divisor $Q_V\in |2B+(g-k)R|$ containing $R_V$. Moreover we have
$$
Q_{V|C}=B_{|C}+R_V,
$$
and $B$ is not a component of $Q_{V}$.
\end{lemma}

\begin{proof}
By observing that $-C+2B+(g-k)R\sim -A$ one can easily see that the restriction morphism
$|2B+(g-k)R|\to |B_{|C}+R_V|$ is an isomorphism. 

To prove the second claim, we note that $-C+B=-(2B+(2g-2-3k)R)$. Since $2B+(2g-2-3k)R$ is big and nef, then $h^1(S,\cO_S(-C+B))=0$. On the other hand we have $h^0(S,\cO_S(B))=1$, therefore $h^0(C,\cO_C(B_{|C}))=1$, which concludes the proof.
\end{proof}

Next we consider the linear system
$|2B+(2g-3k-2)R|$.

\begin{lemma}\label{altrequadriche} The restriction map $|2B+(2g-3k-2)R|\to |B_{|C}+R_V+(g-2k-2)L|$ is an isomorphism.
\end{lemma}
\begin{proof} Note that $2B+(2g-3k-2)R-C=-B$. The claim easily follows since $B$ is an irreducible curve and $S$ is a regular surface.
\end{proof}

Next we note that the sublinear system $R_V+|B_{|C}+(g-2k-2)L|$ of $|R_V+B_{|C}+(g-2k-2)L|$ has dimension 
\begin{equation}\label{quantisono} 
\dim (R_V+|B_{|C}+(g-2k-2)L|)=
g-2k-1.
\end{equation}
Indeed, we have $K_C=H_{|C}\sim B_{|C}+(g-2-k)L$. Hence 
$$
h^1(C,\cO_C( B_{|C}+(g-2k-2)L))=h^0(C,\cO_C(kL))=k+1.
$$ 
By Riemann Roch for curves the claim follows.

Now we consider the sublinear system $\Lambda'<|2B+(2g-3k-2)R|$ on $S$ which is isomorphic to the sublinear system $R_V+|B_{|C}+(g-2k-2)L|$ on $C$. 

\begin{corollary} \label{ilsistemachiave}
There exists $Q_1\in \Lambda'$ such that 
$$
\Lambda'=\langle Q_V+|(g-2k-2)R|, Q_1\rangle.
$$
\end{corollary}
\begin{proof} Note that $Q_V+|(g-2k-2)R|$ is a $(g-2k-2)$-dimensional sublinear system of $\Lambda$. Hence by (\ref{quantisono}) the claim follows.
\end{proof}

\begin{proposition}
The divisors $Q_V$ and $Q_1$ have no common component.
\end{proposition}

\begin{proof}
 Assume by contradiction that there exists a component $\Gamma\in |aB+bR|$ such that $\Gamma<Q_V$ and $\Gamma<  Q_1$ where $g>2k+2$ and $k>1$.

We observe that $\Gamma$ can not be a bisecant divisor on $S$. Indeed, we first note that we can't have
$
\Gamma =Q_V
$,
since $Q_1 \neq Q_V + (g-2k-2)R$ by Corollary \ref{ilsistemachiave}.

Assume now that
\begin{equation}\label{bisec}
\Gamma \sim 2B +bR, \quad Q_V = \Gamma +(g-k-b)R, \quad Q_1 = \Gamma +(2g-3k-2-b)R
\end{equation}
for some $b \le g-k-1$.
As $R_V$ is a subscheme of both $Q_V$ and $Q_1$ and as $R^2=0$, it
would follow that $R_V$ is necessarily a subscheme of $\Gamma$.

On the other hand, we have $Q_{V|C} =B_{|C} +R_V$ by construction, so the relations in (\ref{bisec})
would imply
that 
the subscheme $(g-k-b)R_{|C}$ is contained in $B_{|C}$, which is a contradiction.

Next we claim that we have the following bounds:
\begin{enumerate}\label{trecasi}
\item if $a=0$, then $2g-4k-4\geq b$;
\item if $a=1$, then $b> k+2$.
\end{enumerate}

Indeed, as $R_V$ is the base locus of $R_V+|B_{|C}+(g-2k-2)L|$, we have in particular $R_V>\Gamma_{|C}$.
Consider the sublinear system 
$$
\Lambda'':=\langle (Q_V-\Gamma)+|(g-2k-2)R|, (Q_1-\Gamma)\rangle.
$$ 
The subscheme $R_V-\Gamma_{|C}$ of $C$ is contained in the base locus of $\Lambda''$. Hence
we have 
$$
\Gamma\cdot C+ (Q_V-\Gamma)^2\geq 2g+4.
$$

Assume now $a=0$ and $2g-4k-4\geq b$. In this case $\Gamma\in |bR|$. This would imply that there exists a fiber $R$ of $S\to\mathbb P^1$ such that the subscheme $R_V$ of $C$ contains the subscheme $R_{|C}$,
that is a divisor in the trigonal series. This is impossible.

Finally, assume $a=1$ and $b> k+2$. In this case $Q_V=U+\Gamma$ where $U\in |B+(g-k-b)R|$. 
Since $b>k+2$, then $U=B+U'$ where $U'\in |(g-k-b)R|$. This implies that $B<Q_V$. Then $R_V=(Q_V-B)_{|C}$ and this contradicts Lemma \ref{grandeGianPietro}. 

\end{proof}

\begin{corollary} \label{kappamaggiore}
The subscheme $R_V$ is a complete intersection of $Q_V$ and $Q_1$.

As a consequence we have ${\rm{Bs}}(\Lambda)=R_V$.
\end{corollary}
\begin{proof} 
We have $(2B+(g-k)R)\cdot(2B+(2g-3k-2)R)=2g+4$. Since $R_V$ is a subscheme both of $Q_V$ and $Q_1$ and it is of length $2g+4$, the claim follows.

Finally, as $\Lambda'$ can be embedded in $\Lambda$ by multiplication of a base point free linear system we have the last assertion.
\end{proof}


\subsection{Proof of Theorem \ref{k1}}
In this section we shall treat the remaining cases $k=1$ and $g=6,7$.

\begin{proposition}\label{casosette} If $g=7$ and $k=1$ then the base locus od $\Lambda$ satisfies
$$
{\rm Bs} (\Lambda )= \Gamma \sim 2B+6R.
$$
\end{proposition}
\begin{proof} In this case we have $S \cong \mathbb F_3$, $H \sim B +4 R$ and $C\sim 3B+9R=3A$,
where $A \sim B+3R$ denotes the tautological divisor of $\mathbb F_3$.

Note that $\deg\ R_V=2g+4=18=C\cdot 2A$. Since $2A_{|C}\equiv K_C +2L$, and $H^1(S,\cO_{S}(-A))=0$,
 by computing the cohomology of
$
0\to\cO_{S}(-A)\to\cO_{S}(2A)\to\cO_{C}(K_C+2L)\to 0
$
it easily follows that it exists a unique $\Gamma\in |2A|$ such that $\Gamma_{|C}=R_V$. 
\end{proof}

We can describe very explicitly the linear subspace $T<H^1(C,T_C)$. Let $\langle t_0, t_1\rangle=H^0(S,\cO_S(R))$ be a basis of the pencil $\pi\colon S\to\mathbb P^1$ and let $X_1\in H^0(S,\cO_S(A))$ be an irreducible section. If $X_{\infty}\in H^0(S,\cO_S(B))$ then any trigonal curve $C$ can be written with an equation of the type
\begin{equation}\label{eqaziondi C}
C:=( X_{1}^{3}+\alpha_{3}(t_0,t_1)X_{1}^{2} X_{\infty}+\alpha_{6}(t_0,t_1)X_{1}X_{\infty}^2+ \alpha_{9}(t_0,t_1)X_{\infty}^{3}=0)
\end{equation}
where $\alpha_{j}(t_0,t_1)\in\mathbb C[t_0,t_1]_{[j]}$ are general homogeneous polynomials of degree $j=3,6,9$,
so that $C$ is smooth. A simple computation shows that 
\begin{equation}\label{eqaziondi GAMMA}
\Gamma:= ( 3X_{1}^{2}+2\alpha_{3}(t_0,t_1)X_{1} X_{\infty}+\alpha_{6}(t_0,t_1)X_{\infty}^2=0)
\end{equation}
\begin{remark} By Proposition \ref{casosette}
we can write:
$$
T=\langle I_2, \mathbb C[t_0,t_1]_{[2]}\cdot\Gamma\rangle^{\perp}.
$$
\end{remark}


\begin{proposition}\label{casosei} 
If $g=6$ and $k=1$ then we have
$$
{\rm Bs}\ \Lambda = \Gamma \sim 2B +5R.
$$
\end{proposition}
\begin{proof} In this case we have $S\cong \mathbb F_2$, $H \sim B +3 R$, $C\sim 3B+7R$.
This case differs from the analogue case where $g=7$ because all quadrics vanishing on $R_V$ actually vanish also on a scheme of length $17$ on $C$; note that $16={\rm{deg}}R_V$. We observe that the point $B\cap C=\{p\}$ is a subscheme of 
$
\Gamma_{|C}=p+R_V.$
So we can conclude in a similar way as in Proposition \ref{casosette}.
\end{proof}
 \begin{remark} Note that if $g=6$ and $k=1$, by Proposition \ref{casosei} we can write: 
 $$
T=\langle I_2, \mathbb C[t_0,t_1]_{[1]}\cdot\Gamma\rangle^{\perp}.
$$
\end{remark}


 \section{Hyperelliptic families and trigonal deformations}

 Let $\cY\hookrightarrow \cM_g$ be a closed irreducible subvariety where $g\geq 5$ and ${\rm{dim}}\cY=2g-1$. Assume that for a very general $[C]\in\cY$ there exists a dominant morphism $J(D)\twoheadrightarrow J(C)$ where $[D]$ belongs to the hyperelliptic locus $\cH_{g'}$, $g'\geq g$. By standard arguments we can assume the existence of a family of surjective maps of Jacobians:
 \begin{equation}\label{mapsofjacobian}
\begin{tikzcd}
\cJ{\mathcal{D}} \arrow[rr,"f"] \arrow[dr,swap,"\rho"] && \cJ\cC \arrow[dl,"\rho'"] \\
& \cU
\end{tikzcd}
 \end{equation}
such that the moduli map $\Phi\colon \cU\to\cM_g$ induces a generically finite dominant map $\cU\to\cY$. Moreover we can also assume that $f_u\colon J(D_u)\twoheadrightarrow J(C_u)$ and $[D_u]\in \cH_{g'}$, for every $u\in\cU$. Let $W_u<H^0(D_u,\omega_{D_u})$ be the isomorphic image of $H^0(C_u,\omega_{C_u} )$ via the codifferential of $f_u$. In \cite[Proof of Theorem 1.6, p. 13]{NP} they show that if $C_u$ is not hyperelliptic there exists a rational dominant map $D_u\dashrightarrow Z\subseteq  |W_u|^{*}=\bP(H^{1}(C_u,\cO_{C_u}))$ {\it{where $Z$ is a curve contained in the locus of rank $1$ trigonal deformations of $C_u$}}.

\begin{proposition}\label{quasifatto}
If $g\geq 8$ or $g=6,7$ and $k>1$ then $\cH_g$ is the unique closed irreducible subvariety ${\mathcal Y}\subset\cM_g$ of dimension $2g-1$ such that for its generic element $[C]\in {\mathcal Y}$ there exists $[D]\in\cH_{g'}$ such that $J(D)\twoheadrightarrow J(C)$.
\end{proposition}
\begin{proof} Assume that the general $[C]\in\cY$ is not hyperelliptic. By \cite[Theorem]{NP} it follows that $C$ is trigonal. In particular the curve $Z$ recalled above is a curve contained inside the fix part of $\Lambda$. This 
contradicts Corollary \ref{kappamaggiore}. 
\end{proof}

\subsection{Rational curves of rank-$1$ trigonal deformations} 
If $g=6,7$ and $k=1$ there can exist rational curves in the locus of rank $1$ trigonal deformations, but we claim that they cannot be non degenerate. 
By \cite[Theorem]{NP} and by Proposition \ref{quasifatto}, to show our claim, we have to study the rational curves inside the schematic intersection 
$$
\Gamma^{1}:=\Gamma\cap \bP(T^{1}_{C})
$$
where $g=6,7$. Note that $\Gamma^{1}$ must be a proper subscheme of $\Gamma$.

\begin{proposition}\label{ilcasettosette} If $g=7$, $k=1$ and $C$ is generic then $\Gamma$ is smooth, irreducible and $\Gamma^{1}$ is a finite scheme. Moreover if $\Gamma$ is reducible then $\Gamma^{1}$ does not contain any non degenerate curve.
\end{proposition}
\begin{proof} By Proposition \ref{casosette} and the explicit equations (\ref{eqaziondi C}) and (\ref{eqaziondi GAMMA}) the first claim follows. Assume now that $C$ is smooth but non generic and that $\Gamma$ 
is a union of at least two components. By Lemma \ref{grandeGianPietro}, $\Gamma$ cannot contain $B$ as one of its components. Hence $\Gamma= D_1+D_2$ where $D_1,D_2\in |A|$ since $R_{|C}$ is not a subdivisor of $R_V$. This implies that both components of $\Gamma$ are degenerate curves for the embedding $\phi_{|H|}\colon S\to\bP^6$. In particular  $\Gamma^1=\bP(T^{k}_{C})_{|\Gamma}$ does not contain non degenerate rational curves.
\end{proof} 

\begin{proposition}\label{ilcasettosei} If $g=6$, $k=1$ then either $\Gamma$ is irreducible or if it is reducible it does not contain any non degenerate rational curve. In particular $\Gamma^1$ does not contain any non degenerate rational curve.
\end{proposition}
\begin{proof} The proof is similar to the one of Proposition \ref{ilcasettosette} by using Proposition \ref{casosei}.
\end{proof}

\subsection{The proof of the main Theorem \ref{main}}
By Proposition \ref{quasifatto} we have to consider only the cases where $k=1$, $g=6,7$. We consider the diagram (\ref{mapsofjacobian}). Note that $Z$ is a nondegenerate curve since it is obtained by projection of the canonical image of $D$ which is a rational normal curve since $D$ is hyperelliptic. On the other hand we also have $Z\hookrightarrow \Gamma^1$. By Proposition \ref{ilcasettosette} and by Proposition \ref{ilcasettosei} we see that such a non degenerate curve $Z$ cannot exist.

\end{document}